\documentclass{amsart}    %% Revised version of June 23, 2019 after several sessions beyond receipt of 2nd ref rpt
\usepackage{amsmath,amscd,amsfonts,amssymb,amsthm, multirow, newlfont,exscale,color,mathrsfs, multicol}

\title{Small Subalgebras of Polynomial Rings\\ and Stillman's Conjecture}
\author{Tigran Ananyan and Melvin Hochster$^1$}
\date{\today}
\usepackage{enumerate}
\usepackage{bm}
\theoremstyle{plain}
\newtheorem{theorem}{Theorem}[section]
\newtheorem*{theorema}{Theorem A}
\newtheorem*{theoremb}{Theorem B (existence of small subalgebras)}
\newtheorem*{theoremc}{Theorem C}
\newtheorem*{theoremd}{Theorem D}
\newtheorem*{theoreme}{Theorem E}
\newtheorem*{theoremf}{Theorem F}
\newtheorem{corollary}[theorem]{Corollary}
\newtheorem*{corollarya}{Corollary A}
\newtheorem*{corollaryb}{Corollary B}
\newtheorem*{corollarye}{Corollary E}
\newtheorem{proposition}[theorem]{Proposition}

\theoremstyle{remark}
\newtheorem{remark}[theorem]{Remark}

\newtheorem{discussion}[theorem]{Discussion}
\theoremstyle{definition}

\newcommand{\A}{\mathbb{A}}

\newcommand{\di}{\hbox{dim}}

\newcommand{\inc}{\subseteq}
\newcommand{\ch}{\mathrm{char}}

\newcommand{\cA}{\mathcal{A}}
\newcommand{\cB}{\mathcal{B}}

\newcommand{\cD}{\mathcal{D}}
\newcommand{\cE}{\mathcal{E}}
\newcommand{\cF}{\mathcal{F}}
\newcommand{\cL}{\mathcal{L}}
\newcommand{\cP}{\mathcal{P}}
\newcommand{\cQ}{\mathcal{Q}}

\newcommand{\cT}{\mathcal{T}}

\newcommand{\cV}{\mathcal{V}}

\newcommand{\cG}{\mathcal{G}}
\newcommand{\cI}{\mathcal{I}}

\newcommand{\cM}{\mathcal{M}}
\newcommand{\cN}{\mathcal{N}}

\newcommand{\etA}{{}^\eta\!A}
\newcommand{\etcA}{{}^\eta\!\!\cA}
\newcommand{\etB}{{}^\eta\!B}
\newcommand{\etcB}{{}^\eta\cB}

\newcommand{\etuA}{{}^\eta\!{\underline{A}}}

\newcommand{\fra}{\textrm{frac}}

\newcommand{\height}{\textrm{height}}
\newcommand{\Ker}{\textrm{Ker}}

\newcommand{\N}{\mathbb{N}}

\newcommand{\PP}{\mathbb{P}}
\newcommand{\ov}{\overline}

\newcommand\rom[1]{\text{\textnormal{#1}}}  %% Keeps text in Roman font in theorem environment
\newcommand{\Spec}{\textrm{Spec}}
\newcommand{\surj}{\twoheadrightarrow}

\newcommand{\uA}{\underline{A}}

\newcommand{\Z}{\mathbb{Z}}

\newcommand\vect[2]{#1_1,\,\ldots,\, #1_{#2}}

\def\vect#1#2{{#1}_1, \, \ldots, \, {#1}_{#2}}

\newcommand{\mx}{\begin{pmatrix}}
\newcommand{\emx}{\end{pmatrix}}

\def\todo#1
{\par\vskip 7 pt \noindent {\textcolor{yellow}\footnotesize \color{yellow} \fbox{\parbox{4.8in}{\color{black}
\textbf{To do: } {\color{blue}#1}}}} \vskip 3 pt \par}
\def\forth#1
{\par\vskip 7 pt \noindent {\textcolor{yellow}\footnotesize \color{yellow} \fbox{\parbox{4.8in}{\color{black}
\textbf{For thought: } {\color{blue}#1}}}} \vskip 3 pt \par}

\begin{document}

\begin{abstract}  Let $n, d, \eta$ be positive integers.
 We show that in a polynomial ring $R$ in $N$ variables over an algebraically closed field $K$ 
of arbitrary characteristic, any $K$-subalgebra of $R$ generated over $K$ by at most $n$ forms of 
degree at most $d$ is contained in a $K$-subalgebra of $R$ generated by  $B \leq \etcB(n,d)$ 
forms  $\vect G B$ of degree $\leq d$, where $\etcB(n,d)$ does not depend on $N$ or $K$,
such that these  forms are a regular sequence and such that for any ideal $J$ generated by forms that
are in the $K$-span of $\vect G B$, the ring $R/J$ satisfies the Serre condition $\rom{R}_\eta$.  
%%The proofs depend on giving a very special criterion for $R/I$,  where $I$ is generated by $n$ forms of 
%%degree at most $d$, to satisfy $\rom{R}_\eta$:
%%there is a function $\etcA(n,d)$, independent of $K$ and $N$, such that if no homogeneous generator of $I$
%%is in an ideal generated by $\etcA(n,d)$ forms of strictly lower degree, then $R/I$ satisfies $\rom{R}_\eta$. 
These results imply 
a conjecture of M.~Stillman asserting that the projective dimension of an $n$-generator ideal $I$ of $R$ whose 
generators are forms of degree $\leq d$ is bounded independent of $N$.  We also show that there is a primary decomposition of $I$ 
such that all numerical invariants of the decomposition (e.g., the number of primary components and the degrees and numbers of 
generators of all of the prime and primary ideals occurring) are bounded independent of $N$.    \end{abstract}



\thanks{$^1$The second author was partially supported by grants from the National Science Foundation (DMS--0901145 and
DMS--1401384).}

\maketitle

\pagestyle{myheadings}
\markboth{TIGRAN ANANYAN AND MELVIN HOCHSTER}{SMALL SUBALGEBRAS OF POLYNOMIAL RINGS  AND  STILLMAN'S CONJECTURE}

\section{Introduction}\label{intro}

Throughout this paper, let $R= K[x_1, \ldots , x_N]$ denote a polynomial ring over a field $K$.  We usually assume that 
$K$ is algebraically  closed, but the results bounding projective dimension do not need this restriction.
Stillman's conjecture asserts that given a specified number $n$
of forms of specified positive degrees, say at most $d$, there is a bound for the projective dimension of
the ideal $I$ the forms generate that depends on $n$ and $d$ but not on the number $N$ of variables. 
The conjecture is recorded in \cite{PS} and
previous work related to it may be found in \cite{AH1}, where the problem is solved for quadrics, and in \cite{BS, CK, E1,E2,E3, HMMS1, HMMS2, Mc, McS},
where bounds are given for small numbers of quadrics and cubics and examples are given, and the degree
restriction is shown to be needed, based on 
much earlier work in \cite{Br, Bu, Ko}.  We prove Stillman's conjecture in 
a greatly strengthened form, as well as many other results, e.g., Theorems A, B, C, D, E, and F below.  
In fact, we prove that the forms 
are in a polynomial  
$K$-subalgebra generated by a regular sequence with at most $\cB(n,d)$ elements, where $\cB(n,d)$ does not depend
on $K$ or $N$:  we refer to this smaller polynomial ring informally as a ``small" subalgebra.   

In \cite{AH2} a number of  bounds for degrees 2, 3, and 4 are computed.  While some of the arguments
depend on results of this paper, substantially different techniques are used, particularly for the degree 4 case.
One of the reasons that new methods are needed is that we have not been able to make the results of
\S\ref{bound} in this paper sufficiently constructive.   See Remarks \ref{AH2ad1}  and \ref{AH2ad2} for more specific 
information about bounds from \cite{AH2}.

For the purpose of proving Stillman's conjecture one can pass to the case where the field is algebraically
closed, and we shall assume that $K$ is algebraically closed, unless otherwise stated, throughout the rest of
this paper.

We use $\N$ to denote the nonnegative integers and $\Z_+$ the positive integers.
We define a nonzero homogeneous polynomial $F$ of positive degree in $R$ to have a $k$-{\it collapse}
for  $k \in \N$, if $F$  is in an ideal generated by $k$ homogeneous elements of strictly
smaller positive degree, and we define $F$ to have {\it strength} $k$ 
if it has a ${k+1}$-collapse but no
$k$-collapse.   We shall also say that $F$ is $k$-{\it strong} if it has no $k$-collapse, which means that its
strength is at least $k$.
Because nonzero linear forms do not have a $k$-collapse for any $k \in \N$, we make
the convention that such a form has strength $+\infty$.  A form has strength at least 1
if and only if it is irreducible.  One of the main themes here is that $F$ has a ``small" collapse if and only
if the singular locus of $F$ has ``small'' codimension.   ``Only if" is evident:  when $F = \sum_{i=1}^k G_iH_i$,
the partial derivatives of $F$ are in the $2k$-generated ideal $(G_i,\, H_i:1 \leq i \leq k)R$.  ``If" is
quite difficult:  a precise statement is made in part (a) of Theorem A below. 

We use $V$ to denote a finite-dimensional graded vector subspace of $R$ spanned by forms of positive
degree.  If $d$ is an upper bound for the degree of any element of $V$, we may write 
$V = V_1 \oplus \ldots \oplus V_i \oplus \ldots \oplus V_d$, where $V_i$ denotes the $i\,$th graded piece,  
and we shall say 
$V$ has {\it dimension sequence} $(\vect \delta d)$ where $\delta_i := \di_K(V_i)$.  This sequence
carries the same information as the Hilbert function of $V$.  
 We regard two such dimension sequences
as the same if they become the same after shortening by omitting the rightmost string of consecutive 0 entries. 

For $V$ as in the preceding paragraph, we say that $V$ has {\it strength} at least $k$ or is $k$-{\it strong} if every nonzero homogeneous
element of $V$ is $k$-strong.

If $F$ is a form of degree $d$ in $K[\vect x N]$, we denote by $\cD F$ the $K$-vector space spanned by the partial derivatives  $\partial F/\partial x_i$,  $1 \leq i \leq N$.  When the characteristic does not
divide $d$, we have that $F \in (\cD F)$, the ideal generated by $\cD F$, since Euler's formula 
asserts that $$\deg\,(F)F = \sum_{i=1}^N x_i (\partial F/\partial x_i).$$

If $\sigma$ is a subset of a polynomial ring $R = K[\vect x N]$, where $K$ is an algebraically closed field, we write 
$\cV(\sigma)$ for the algebraic set in $\A^N_K$ where the elements of $\sigma$ all vanish.  

We recall that a Noetherian ring $S$ satisfies the Serre condition $R_\eta$,  where $\eta$ is a nonnegative integer, if
$S_P$ is regular for every prime $P$ of height $\leq \eta$.  If the singular locus of $S$ is closed and defined by an ideal $J$,
this is equivalent to the condition that $J$ have height at least $\eta +1$.  (By convention, the unit ideal has  height $+\infty$.)
We shall say that a sequence of elements generating a proper ideal of a ring $S$  is a {\it prime} sequence 
(respectively, if $S$ is Noetherian, an  $\rom{R}_\eta$-sequence,
where $\eta \in \Z_+$), if the quotient of $S$ by the ideal generated by any initial segment is a domain (respectively,
satisfies $\rom{R}_\eta$). A prime sequence in $S$ is always a regular sequence.  If $S = R$ is a polynomial ring, 
every $\rom{R}_\eta$-sequence of forms of positive
degree for $\eta \geq 1$ is a prime sequence (in fact, the
quotients are normal domains), and, hence, a regular sequence.  Note that if the regular sequence is such that 
the successive quotients are normal, then the sequence must be an R$_1$-sequence.

We call a function $\cB:\N^ h \to \Z$ {\it ascending} if it is nondecreasing in each input
when the others are held fixed.  In all our constructions of  functions, it is easy to make them ascending:
replace the function  $\cB$ by the one whose value on $(\vect b h)$ is $\max\{\cB(\vect a h): 0 \leq a_i \leq b_i
\hbox{ for } 1\leq i \leq h\}$.  A $d$-tuple of integer-valued functions on $\N^h$ will be called {\it ascending}
if all of its entries are ascending functions. 

The main results are stated below.  The proofs are given in \S\ref{proof}, after some preliminary results are
established in \S\ref{lemmas} and \S\ref{bound}.

We want to emphasize that in Theorem~A below, the functions $\etA$,  $\uA$, and $\etuA$ {\it do not depend} on the field
$K$ nor on the number of variables $N$.  This fact is central to their use in proving Stillman's conjecture and related
results. 

\begin{theorema} \begin{enumerate}[(a)]
\item There exists an integer $\etA(d) \geq d-1 \geq 0$, ascending as a function of $\eta,\, d \in \Z_+$, such that 
for every algebraically closed field $K$ and for every positive integer $N$, if $R = K[\vect x N]$ is a polynomial
ring and $F \in R$ is a form of degree $d \geq 1$
of strength at least $\etA(d)$, then the codimension of the singular locus in $R/FR$ is at least $\eta +1$, so that $R/FR$ 
satisfies the Serre condition $\rom{R}_\eta$.  

\item There are ascending functions $\uA = (\vect A d)$ and, for every integer $\eta \geq 1$, $\etuA = (\vect \etA d)$ 
from dimension sequences $\delta = (\vect \delta d) \in \N^d$  to $\N^d$ with the following property:

For every algebraically closed field $K$ and every positive integer $N$, if $R = $ $K[\vect x N]$ is a polynomial ring,
and $V$ denotes a graded $K$-vector subspace of $R$ of vector space dimension $n$  
with dimension sequence $(\vect \delta d)$, such that for $1 \leq i \leq d$,  the strength of
every nonzero element of $V_i$ is at least $A_i(\delta)$  (respectively, $\etA_i(\delta)\,$),
then every sequence of $K$-linearly independent forms in $V$ is a regular sequence (respectively,
is an $\rom{R}_\eta$-sequence). 

\item If we have the functions $\etA(i)$ described in (a) for $1 \leq i \leq d$,  we may
take  $\etA_i(\delta) = \etA(i) + 3(n-1)$, where $n = \sum_{j=1}^d \delta_j$,   and these functions 
will have the property described in part (b).
\end{enumerate} 
\end{theorema}

%1.1
\begin{remark} The condition that the singular locus of $R/FR$ have codimension at least $\eta + 1$ in 
$R/FR$,  i.e., that $R/FR$ satisfy the Serre condition $\rom{R}_\eta$, is equivalent to the condition that the 
ideal $FR + (\cD F)R$ have height $\eta + 2$ in $R$.  (If the characteristic is 0 or does not divide $\deg(F)$, 
$F$ is in the ideal $(\cD F)R$.)  \end{remark}
%1.2
\begin{remark}\label{AH2ad1} It would be of great interest to get specific bounds for  
the functions in Theorem A.   In \cite{AH2}, Theorem 4.20 and Corollary 4.21, it is 
shown that if $V$ is a vector space of quadratic forms in $R$ of dimension $n$ over $K$ 
such that every element of $V-\{0\}$ is $(n-1)$-strong then every sequence of linearly independent elements of 
$V$ is a regular sequence.  
If $\eta \geq 1$ and every element of  $V-\{0\}$ is $(n-1 +\lceil \frac{\eta}{2}\rceil)$-strong,  
then the quotient by the ideal generated 
by any elements of $V$ satisfies the Serre condition \rom{R}$_\eta$.  The corresponding result for
a vector space of cubics  of dimension $n$ over $K$ uses a function in the strength condition that is quadratic in $2n+\eta$. See
\cite{AH2}, Theorem 6.4.  One might hope that for degree $d$,  the function in the strength condition needed for a vector space
of dimension $n$ consisting of $d$-forms might be a polynomial of degree $d-1$ in $n$ and $\eta$. So far as we know,
this might be true.   However, the result obtained for quartics in \cite{AH2}, Corollary 10.4, is somewhat worse than exponential.
\end{remark}

By taking a supremum over values of the $\etA_i$ over all dimension sequences with at most $d$ entries
such that the sum of the entries is at most $n$ we have at once the result mentioned in the abstract:

\begin{corollarya} There is an ascending function $\etcA(n,d)$, independent of $K$ and $N$,  such that for all polynomial
rings $R  = K[\vect x N]$ over an algebraically closed field $K$ and all ideals $I$ generated by a graded vector space
$V$ of dimension $\leq n$ whose nonzero homogeneous elements have positive degree at most $d$,  if no homogeneous element of $V - \{0\}$
is in an ideal generated by $\etcA(n,d)$ forms of strictly lower degree, then $R/I$ satisfies $\rom{R}_\eta$. 
\end{corollarya}

We use Theorem A to prove:

\begin{theoremb}\label{small}  
There is an  ascending function $B$ from dimension sequences $\delta = (\vect \delta d)$ to $\Z_+$
with the following property. 
If $K$ is an algebraically closed field and
$V$ is a finite-dimensional $\Z_+$-graded $K$-vector subspace of a polynomial ring $R$ over $K$  with  
dimension sequence $\delta$,  then $V$ (and, hence, the $K$-subalgebra of $R$ generated by $V$) 
is contained in a $K$-subalgebra of $R$ generated by a regular sequence $\vect G s$ of forms of degree at most $d$, where
$s \leq B(\delta)$.

Moreover, for every $\eta \geq 1$ there is such a function $\etB$ with the additional property that every sequence consisting
of linearly independent homogeneous linear combinations of the elements in  $\vect G s$ is an $\rom{R}_\eta$-sequence.
\end{theoremb}
%%1.3
\begin{discussion}\label{groth}  We note, for example, that this theorem implies for $\eta \geq 3$ that all the quotients of $R$ by  
ideals generated by homogeneous
linear combinations of the elements in $\vect G s$ are
unique factorization domains:  this follows at once from a theorem of Grothendieck, conjectured by Samuel,
for which there is an elementary exposition in \cite{Ca}.  \end{discussion}

By taking a supremum over all dimension sequences with at most $d$ entries such that the sum of the entries
is at most $n$,  we have at once:

\begin{corollaryb}
There is an ascending function $\etcB(n,d)$, independent of $K$ and $N$,  such that for all polynomial
rings $R  = K[\vect x N]$ over an algebraically closed field $K$ and all graded vector subspaces $V$ of $R$ of dimension
at most $n$ whose homogeneous elements have positive degree at most $d$, the elements of $V$ are contained
in a subring  $K[\vect G B]$,  where $B \leq \etcB(n,d)$ and $\vect G B$ is an $\rom{R}_\eta$-sequence of forms
of degree at most $d$. \end{corollaryb}

We want to emphasize that in Theorems C, D, E and Corollary E below,  given elements and entries are not
necessarily assumed to be homogeneous:  one obtains the results by passing to a subalgebra that contains all of their
homogeneous components. Note that the degree of a polynomial provides an upper bound for the number of its positive degree
homogeneous components with no reference to the base field nor to the number of variables. \smallskip

Theorem B easily implies a strong form of  M.~Stillman's Conjecture:

\begin{theoremc}\label{matrix}  There is an ascending 
function $C$ from $\Z_+ \times \Z_+ \times \Z_+ \to \Z_+$ with the following property.  If $R$ is a polynomial ring over an arbitrary field $K$ and $M$ is a module that is the 
cokernel of an $m \times n$ matrix whose (not necessarily homogeneous) entries have degree at most $d$,  then the projective dimension of $M$ is bounded by $C(m,n,d)$. \end{theoremc}
%%1.4
\begin{remark}\label{AH2ad2} In \cite{AH2}, Theorem 4.22,  it is shown that a $K$-vector space $V$ of quadrics of 
dimension $n$  in the polynomial ring $R$  is contained in a polynomial subring
generated by a regular sequence consisting of at most $2^{n+1}(n-2)+4$ linear and quadratic forms, and so this number also  
bounds the projective dimension of $R/I$ over $R$,  where $I$ is the ideal generated $V$.  
The corresponding result for cubics is not made explicit in \cite{AH2}:  it is $n$-fold exponential.  See \cite{AH2},
Discussion 6.6.  It may be that much smaller bounds exist, especially for projective dimension.  In the case
of quadrics, the bound for projective dimension may be quadratic in $n$. See  \cite{HMMS1, HMMS2}. \end{remark}

Theorem B yields many other bounds.  In the following theorem we give bounds on a finite free resolution and on a primary
decomposition.  The resolution and the primary decomposition are not unique, so what we mean is that there exists
some finite free resolution and some primary decomposition for which the bounds hold.

\begin{theoremd}\label{bdev} Let $K$ be an algebraically closed field and let $R = K[\vect xN]$ be 
the polynomial ring in $N$ variables over $K$.   Let $m, n, d \in \Z_+$, let $\cM$ be an $m \times n$ matrix over 
$R$ whose (not necessarily homogeneous)  entries have degree at most $d$, let
$M$ be the column space of $\cM$.  
\begin{enumerate}[(a)]

\item There exists an ascending function $P(m,n,d)$, independent of $N$ and $K$, that bounds the length of a finite free resolution of $M$, the ranks of the free modules occurring, and the degrees 
of all of the entries of all of the matrices occurring.  Hence, $P(m,n,d)$ bounds sets of generators for the modules of syzygies associated with the resolution.  
 In the graded case, $P(m,n,d)$ bounds the twists of $R$ that
occur as summands in a minimal free resolution of $M$. 

\item There exists an ascending function $E(m,n,d)$, independent of $N$ and $K$, that bounds
the number of primary components in an irredundant primary decomposition of $M$ in $R^m$, 
the number of and the degrees of the generators of every prime ideal occurring, and 
the number of generators and the degrees of the entries
of the generators for every module in the decomposition.  $E(m,n,d)$ can also be taken to bound the exponent
on every associated  prime ideal  $P$ needed to annihilate the corresponding $P$-coprimary
component  of $M$ mod $M$ (in the ideal case, the exponent $a$ needed so that $P^a$ is contained 
in the corresponding primary ideal of the decomposition). 

\item There exists an ascending function $D(k,d)$, independent of $N$ and $K$, that bounds the minimum number of 
generators of any minimal prime of an ideal generated by a regular sequence consisting of $k$ or fewer $d$-forms.

\end{enumerate}
\end{theoremd}
%%1.3
\begin{remark} Part (c) is obvious from part (b),  since we may take $D(k,d) = E(1,k,d)$.  However,
the function $D(k,d)$ plays a special role in the proofs, and may have a much smaller bound. \end{remark}

Free resolutions are not unique, but the specified bounds work for at least one free resolution.
Similarly, primary decompositions are not unique, but the specified bounds work for at least one irredundant primary
decomposition of $M$ in $R^m$.  Of course, when $m=1$ we are obtaining such a bound for the primary
decomposition of an ideal with $n$ generators when the degrees of the generators are at most $d$.

 We shall
refer to the largest degree of any entry of a nonzero element  $v$ of the free module $R^m$ over the polynomial
ring $R$ as the {\it degree} of $v$.  We shall say that a set of generators for a submodule of $R^m$ is {\it bounded}
by $n$, $d$ if it has at most $n$ elements of degree at most $d$. If $n = d$, 
we say that the set of generators is bounded by $n$.

\begin{theoreme} There exist ascending $\Z_+$-valued functions $\Theta(m,n,r,d)$, $\Lambda(m,n,d,h)$, and 
$\Gamma(m,n,d)$ of the nonnegative integers  $h\geq 2,\,m,\,n,\,r,\,d$ with the following properties. 
Let $R = K[\vect x N]$ be a polynomial ring over an algebraically closed field 
$K$.  Let $G:=R^m$. 
Let $M, \, Q$ and $\vect Mh$ be submodules of $G$.  Let $I$ be an ideal of $R$.  Suppose that
all of $M, Q$, $\vect Mh$, and $I$ have sets of generators bounded by $n$, $d$.
\begin{enumerate}[(a)]
\item Given an $m \times r$ matrix over $R$ with entries of degree at most $d$, thought of as a map from $R^r \to R^m$ and $M \inc R^m$ as above, 
there is a set of generators for
$\Ker (R^r \to R^m \surj R^m/M)$ bounded by $\Theta(m,r,n,d)$.
\item There exists a set of generators for $M_1 \cap \cdots \cap M_h$ bounded by $\Lambda(m,n,d,h)$.
\item There exist
sets of generators for $M:_R Q$, $M:_G I$, and $M:_G I^{\infty}$ bounded by $\Gamma(m,n,d)$. 

\end{enumerate}
\end{theoreme}

\begin{remark} Given a map of finitely presented $R$-modules, we may always think of it as 
induced by a map of free modules that map onto these $R$-modules, so that it may be described
as the map $R^r/M' \to R^m/M$ determined by the $m \times r$ matrix of a map of the free numerators.
The kernel of this map is generated by the images of the generators of the kernel of the map to
$R^r \to R^m/M$.  Thus, part (a) of Theorem~E enables one to bound a set of generators for the kernel of  a
map of finitely presented modules when we have information bounding the sizes and degrees of the 
presentations and of the matrix of the map of free modules. \end{remark} 

\begin{remark}  It is difficult to make a comprehensive statement of all the related results that follow
from the main theorems:  the following is an example. In the result below, by the ``leading form" of a polynomial
we mean the nonzero homogeneous component of highest degree, or 0 if the polynomial is 0.  \end{remark}

\begin{corollarye} Let $R = K[\vect x N]$ be a polynomial ring over an algebraically closed field.  There exist
bounds for the number of generators of the ideal generated by the leading forms of the elements in an
ideal generated by $n$ elements of degree at most $d$ that depend on $n$ and $d$ but not on $N$ or $K$.
\end{corollarye}
\begin{proof}  Let the ideal be $(\vect f n)R$. Let $\vect F n$ be the result of homogenizing the $f_i$ with
respect to an new variable $x = x_{N+1}$.  Then $\vect F n$ also have degree at most $d$,  and the required ideal
is the image of $(\vect F N):_{R[x]}x^\infty$ mod $x$.  \end{proof}

\begin{theoremf} There is an ascending function $\Phi(h,d)$ such that, independent of the algebraically
closed field $K$ or the integer $N$,  if a form $F$ of degree $d$ in the polynomial ring
$K[\vect x N]$ has strength at least $\Phi(h,d)$,  then $\cD F$
is not contained in an ideal generated by $h$ forms of degree at most $d-1$. \end{theoremf}

Of course, this is obvious from Euler's formula if $p:=\ch(K)$ does not divide $d$:  in that case we may
take $\Phi(d,h) = h$, since $F$ is in the ideal $(\cD F)R$.  We handle the case where $p$ is a positive
prime that may divide $d$ inductively, by using the fact that we know Corollary B  for integers less
than $d$. See Proposition~\ref{phi}. \medskip

We end this section with a brief overview of the structure of the proof. 
The main results are proved by simultaneous induction.  \S\ref{lemmas}
through Theorem~\ref{combined} establishes lower bounds on the codimension of the singular locus of a variety
defined by a regular sequence of forms needed in the proofs of statements about the existence of
$\rom{R}_\eta$-sequences.  Proposition~\ref{phi}  is independent of other material in \S\ref{lemmas}:  it plays a key role
in the induction.  The rest of \S\ref{lemmas} is concerned with proving results on when prime (respectively, primary)
ideals retain that property after extension.  

In the course of the induction, one can sometimes pass, using
cases of the theorems that are already known, to a polynomial
subring in which one has a bound for the number of variables.  \S\ref{bound} contains results that provide
other relevant bounds once a bound for the number of variables is known.  

In \S\ref{proof}, all the prior results are combined in the simultaneous induction that yields the proofs of all of
the main theorems.

\section{Preliminary results}\label{lemmas} %%Section 2

The proofs of our main results depend on giving lower bounds for the codimension of the singular locus of the
variety defined by a regular sequence, which means giving a lower bound for the heights of certain ideals generated by maximal  
minors of Jacobian matrices.  Theorem~\ref{codimsing} enables us to reduce to the case where the elements in the regular sequence
have mutually distinct degrees, while Theorem~\ref{MAX} gives a very strong result on heights of  ideals of maximal minors in
the situation where one can assign a degree to every row of the matrix such that all elements of that row have the assigned degree
but distinct rows are assigned distinct degrees.  The desired result on codimension is obtained, using these earlier results,  
in Theorem~\ref{combined}.

Proposition~\ref{phi} is, in a sense, unrelated to other results in this section.  It provides a key step in the complex multiple
induction that simultaneously proves all of our main results in Section~\ref{proof}.

The remaining results in this section are aimed at giving conditions on a faithfully flat extension 
$R \inc S$ so that every prime (and, hence, also, every primary) ideal of $R$ remains prime
(respectively, primary) when extended to $S$.  The case needed for our proofs
is when $S$ is a polynomial ring over an algebraically closed field $K$ and $R$ is generated over $K$ by a prime sequence of forms of positive degree:  this is, essentially,  Corollary~\ref{expand}.  
This is needed in our proof of results bounding primary decomposition, which proceeds by passing from the original ring to a small
subalgebra generated by a homogeneous prime sequence.
%% 2.1
 \begin{theorem}\label{codimsing} Let $K$ be an algebraically closed field, let $R =K[\vect xN]$ be a polynomial
 ring.  Let $V$ be a graded $K$-vector subspace of $R$, say
 $V = V_1 \oplus\,\cdots\, \oplus V_d$, where $V_i$ is spanned by forms of degree $i$, and suppose that
 $V$ has finite dimension $n$.  Assume that a homogeneous basis $\vect  F n$ for $V$ is a regular sequence in $R$.  
 Let $X = \cV(\vect F n)$.  
 Let $S$ be the family of all subsets of $V$ consisting of nonzero forms with mutually distinct degrees, so that
 the number of elements in any member of $S$ is at most the number of nonzero $V_i$. 
 For $\sigma \in S$, let $C_\sigma$ be the codimension of the singular locus of $\cV(\sigma)$
 in $\A^N_K$.  
 Then the codimension in $\A^N_K$ of the singular locus of $X$  is at least $(\min_{\sigma \in S} C_\sigma) - (n-1)$.\end{theorem}
 \begin{proof} We study the codimension of the set where the Jacobian of $X$ has rank at most $n-1$. Let
 $Z$ denote an irreducible component of the singular locus of $X$.  
 We first consider the case where the Jacobian has rank 0 on $Z$, i.e., where
 it vanishes identically.  Let $\lambda_0$ be the set of all $i$ such that $V_i \not=0$.  If we form $\sigma$ by
 choosing one form 
 $G_i$ of each degree $i \in \lambda_0$,  then
 $Z$ is in the singular locus of the scheme $Y = \cV(G_i: i \in \lambda_0)$ defined by the vanishing of these $G_i$ (evidently, the Jacobian
 of this smaller set of polynomials is still identically 0 on $Z$), which shows that the dimension of the singular
 locus of $Y$ is at least as large as the dimension of $Z$, and hence $C_\sigma$ is a lower bound for the
 codimension of $Z$. 
  
 We second consider an irreducible component $Z$ of the singular locus, such that on a nonempty
open subset $U_1$ of $Z$,  the Jacobian matrix has rank $r $, $1 \leq r \leq n-1$.  We can choose an $r \times r$ minor 
$\mu$ of the Jacobian matrix that does not vanish on a dense open subset $U$ of $U_1$, and it will suffice to bound below the codimension of $U$ in $\A^N_K$.  Choose forms which, after renumbering, we may assume are $\vect F {r+1}$ in the basis for $V$ such that the
corresponding  $r+1$ rows of the Jacobian
matrix contain the $r$ rows corresponding to $\mu$. We have a map $\theta:U \to \PP^r$ that assigns to each 
point $u \in U$ the non-trivial relation
on the rows of the Jacobian matrix $J_0$ of $\vect F {r+1}$ when it is evaluated at $u$:  since the $J_0$ has rank
exactly $r$ at $u$,  this relation is unique up to multiplication by a nonzero scalar.  In fact, it is given by the $r \times r$
minors of the $r$ columns determined by the nonvanishing minor $\mu$.  Since the dimension of $\theta(U) \inc \PP^r$ 
is at most $r$,  the dimension of $U$ is bounded by the sum of $r$ and the dimension of a typical fiber $Y$
of the  map.  Note that $r \leq n-1$, and  the codimension of $U$ in $\A^N_K$ is bounded below by $C-r$,  where $C$
is the codimension of a typical fiber of the map $\theta:U \to \PP^r$.  Consider the fiber over the point 
$u = [a_1: \cdots : a_{r+1}]
\in \PP^r$.  Because the $a_i$ give a relation on the rows of the Jacobian matrix corresponding to $\vect F {r+1}$,
it follows that all of the partial derivatives of $F = \sum_{i=1}^{r+1} a_i F_i$ vanish on $U$.  We can break this
sum up as a sum of nonzero forms of mutually distinct degrees, say $F = G_{i_1} + \cdots + G_{i_h}$ where
$1 \leq i_1 < \cdots < i_h \leq d$ are the degrees.  But then the sum of the rows of the Jacobian matrix for 
 $Z_0 = \cV(G_{i_1}, \ldots, G_{i_h})$  vanishes on $U$,  and so $U$ is contained in the singular locus of $Z_0$.
 The codimension in $\A^N_K$ of the singular locus of $Z_0$ is bounded below by $C_\sigma$ with 
 $\sigma =\{G_{i_1}, \, \ldots,\,  G_{i_h}\}$.  Thus the codimension of $U$ in $\A^N_K$ is bounded below by 
 $C_\sigma - r$,
 where $r \leq n-1$. This yields the stated result. \end{proof}

%% 2.2
\begin{remark}\label{Serreht} Note that in a polynomial ring, the height
of a homogeneous ideal $I$ does not increase when we kill some of the variables.  
Let $P$ be a minimal prime of $I$ whose height is the same as that of $I$, and let
$Q$ be the prime generated by the variables we are killing. The result holds because 
we may localize at a minimal prime of $P+Q$, and we may apply the result of 
\cite{Serre}, Th\'eor\`eme 1, part (2), p.~V--13, which implies
that $\height(P + Q) \leq \height(P) + \height(Q)$. We shall make use of this in the proof
of Theorem~\ref{MAX} below.  \end{remark}

%%2.3
\begin{remark}\label{distinct} In the theorem just below, the hypothesis that the degrees associated
with the various rows be distinct is crucial:  without it, the rows could all be taken to be the same.
Having the degrees be all different somehow makes the matrix more like a generic matrix, i.e., a 
matrix of indeterminates, for which results like the one below have long been known:  cf.~\cite{EN}, \cite{HE}.
 \end{remark}

%%2.4
\begin{theorem}\label{MAX} Let $K$ be a field, let $R$ be a polynomial ring over $K$, and let
$M$ be an $h \times N$ matrix such that for $1 \leq i \leq h$,  the $i\,$th row consists of forms
of degree $d_i \geq 0$ and the $d_i$ are mutually distinct integers.  Suppose that for $1 \leq i \leq h$,
the height of the ideal generated by the entries of the $i\,$th row is at least $b$. (If the row consists
of scalars, this is to be interpreted as requiring that it be nonzero.)  Then the ideal
generated by the maximal minors of the matrix has height at least $b - h + 1$.  \end{theorem}

\begin{proof} Without loss of generality, we may enlarge the field to be algebraically closed, 
and we may assume that $d_1 < \, \ldots \, < d_h$.  
We use induction on $h$:  the case where $h = 1$ is immediate. (If one has a single nonzero
row of scalars, the height of the ideal generated by the maximal minors is $+\infty$.)  We therefore assume $h \geq 2$
and that the result holds for smaller $h$.  Next, we reduce to the case where the number of variables
in $R$ is $b$, and every non-scalar row generates an ideal primary to the homogeneous maximal ideal.  Suppose 
that the number of variables
is greater than $b$.  For each $i$, choose a subset of the span of the $i\,$th row generating an ideal $J_i$
of height $b$. Choose a linear form that is not in any of the minimal primes of any of the $J_i$.  
We may kill this form, and the hypotheses are preserved:  the height of the ideal generated by the maximal
minors does not increase by Remark~\ref{Serreht}.  We may continue in this way until the
number of variables is $b$.  

Let $P$ be a minimal prime ideal of the ideal generated by the maximal minors of $M$.
To complete the proof, it will suffice to show that the dimension of the ring $R/P$ is at most $h-1$.

Let $\ov{M}$ denote the image of the matrix $M$ over $R/P$. 
It is possible that all of the maximal minors of the matrix formed by a proper subset consisting
of $h_0 < h$  of the rows
of $\ov{M}$ vanish in $R/P$.  But then the height of the ideal generated by the maximal minors of these rows is at least  
$b-h_0 + 1$ by the induction hypothesis, and this shows that the dimension of $R/P$ is at most $h_0-1$.
Hence, we may assume that there is 
no linear dependence relation on any proper subset of the rows of $\ov{M}$, while the rank of the image
$\ov{M}$ is $h-1$.  This implies that there are unique elements of the fraction field of $R/P$, call them
$u_1, \ldots, u_{h-1}$, such that $\rho_h = \sum_{i=1}^{h-1} u_i \rho_i$, where $\rho_i$ is the image
of the $i\,$th row of $M$.  More specifically, since the first $h-1$ rows of $\ov{M}$ are linearly independent
over $\fra(R/P)$, we may choose $h-1$ columns forming an $h \times (h-1)$ submatrix $\ov{M}_0$ of $\ov{M}$ 
such that the $h-1$ size minor $\Delta$ of the first $h-1$ rows is not 0.
The nonzero relation, unique up to multiplication by a nonzero scalar in $\fra(R/P)$,  
on the rows of the submatrix $\ov{M}_0$ is given by the vector whose entries are its 
$h-1$ size minors, which
are homogeneous elements of $R/P$. This must give the relation on the rows of $\ov{M}$. 
 Thus, every $u_j$ can be written as a fraction with denominator $\Delta$
whose numerator is one of the other minors of $\ov{M}_0$. 

Let $S$ be the ring
$(R/P)[\vect u {h-1}]$.  Note that $u_i$ has degree  $d_h - d_i > 0$, so that $S$ is a finitely generated 
$\N$-graded $K$-algebra  with $S_0 = K$  
generated over $K$ by the images of the $x_i$ and by the $u_i$.  The Krull dimension
of $S$ is the same as that of $R/P$, since the fraction field has not changed, and that is the same as the
height of the maximal ideal of $S$.  But $S/(\vect u {h-1})S$ is zero-dimensional, since the vanishing of the
$u_i$ implies the vanishing of all entries of $\rho_h$,  and these entries generate an ideal primary to the homogeneous
maximal ideal of $K[\vect x b]$.  It follows that the Krull dimension of $S$ is at most $h-1$, 
and, hence, the same holds for $R/P$, as required. \end{proof}

%%2.5
\begin{theorem}\label{combined} Let  $K$ be an algebraically closed field, and let $V$ be an $n$-dimensional 
graded $K$-vector subspace of the polynomial ring $R = K[\vect x N]$ consisting of forms of  degree between $1$ and 
 $d$, so that $V = V_1 \oplus \cdots \oplus V_d$. Assume that a basis for $V$ consisting of forms is a regular
 sequence in $R$.   Let $h$ denote the number of integers $i$ such that $V_i \not=0$, so that
 $h \leq \min\,\{d,\,n\}$.  Suppose that for every nonzero homogeneous element $F$ of $V$, the
 height of the ideal $(\cD F)R$ in $R$ is at least $\eta + h+2n - 1$.  Then the codimension of the singular locus
 of  $R/(V)R$  in  $R/(V)R$  is at least  $\eta + 1$.
 
\end{theorem} 
\begin{proof}  Consider any set $\sigma$ of homogeneous elements of $V$ of distinct degrees:
it has at most  $h$ elements.  The Jacobian matrix of the elements of $\sigma$ has at most
$h$ rows, and the degrees associated with the rows are distinct.  By hypothesis, each row
generates an ideal of  height $\eta + h+2n-1$ in $R$.  By Theorem~\ref{MAX},  the height of the
ideal of maximal minors is at least $\eta + 2n-1 + 1$.  Hence, the codimension $C_\sigma$
of the singular locus of $V(\sigma)$ in $\A^N_K$ is at least $\eta+2n$.  By 
Theorem~\ref{codimsing},  the codimension of the singular locus of $R/(V)R$ in  $\A^N_K$ is at least
$\eta + n +1$.  When we work mod $(V)R$ this codimension can drop, at worst, to $\eta +1$.  
\end{proof}

The following result shows that Corollary B in degree $d-1$ implies Theorem F in degree $d$. 
%%2.6
\begin{proposition}\label{phi} Suppose that we have a function $^3\cB(n,d-1)$ for a fixed value of $d$ and all $n$, as in the statement 
of Corollary B.  Then Theorem F holds
with $\Phi(h,d) =  {}^3\cB(h, d-1)+1$. \end{proposition}

\begin{proof}   Suppose that a form $F$ of degree $d$ in $K[\vect x N]$ has strength at least $^3\cB(h,d-1)+1$
but that  $\cD F$ is contained in the ideal generated by $h$ forms of degree $d-1$ or less.  By Corollary B
these forms are contained in a subring $K[\vect G B]$  where $B \leq {}^3\cB(h,d-1)$ and
$\vect G B$ form an  $\rom{R}_3$-sequence.  Then $\cD F$ is also contained in the ideal generated
by $\vect G B$.  Since $R/(\vect G B)$ is a complete intersection that is $\rom{R}_3$,  it is a UFD, by Discussion~\ref{groth}.
$F$ must be irreducible in this quotient, or else we obtain a homogeneous equation
$F = F_1 F_2 + \sum_{i=1}^B G_i H_i$.  Thus, $F$ has a $(B+1)$-collapse, contradicting
the hypothesis.  Therefore,  $\vect G B, \, F$ is a prime sequence.  This implies that the maximal
minors of the Jacobian matrix of $\vect G B, \, F$ generate an ideal of positive height mod $(\vect G B,\, F)R$.  Hence the row of
the Jacobian matrix corresponding to $F$, whose $K$-span is $\cD F$,  cannot be 0 mod $(\vect G B)$. \end{proof}

\subsection*{Extension of prime ideals} 
Recall that a flat ring homomorphism $R \to S$ is {\it intersection flat} if for every family $\cI$ of ideals
of $R$,  $\bigcap_{I \in \cI} (IS) = (\bigcap_{I \in \cI} I)S$. Flatness implies this condition when $\cI$ is 
a finite family.  By \cite{HH}, p.~41, if $S$ is free over $R$, then $S$ is intersection flat.  
In the situation where $\vect G B$ is part of a homogeneous system of parameters for the polynomial
ring $K[\vect x N]$,  if $\vect G N$ is a full homogeneous system of parameters we know that
we have free extensions $K[\vect G B] \to K[\vect G N]$ and $K[\vect G N] \to K[\vect x N]$ (this is module-finite
and free, since the target ring is Cohen-Macaulay). Moreover, if $K \inc L$ is a field extension,
$K[\vect x N] \to L[\vect x N]$ is free, since $L$ is free over $K$. Hence, $K[\vect G B] \to L[\vect x N]$ is
free and, consequently, intersection flat.

Recall also that $R$ is a {\it Hilbert ring} if every prime ideal is an intersection of maximal ideals.
 
We first observe the following:
%%2.7
\begin{theorem}\label{primeHilb} Let $R$ be a Noetherian Hilbert ring, and let $S \supseteq R$ be a Noetherian $R$-algebra that is intersection flat over $R$. 
Suppose that for every maximal ideal $m$ of $R$,  $S/mS$ is a domain.  Then for every prime ideal
$P$ of $R$,  $S/PS$ is a domain.  \end{theorem}
\begin{proof} Suppose the theorem is false and that $P$ is maximal among the primes in $R$ that give a counterexample.  
The case where $R/P$ has dimension 0 is the hypothesis.
Now assume that $\di(R/P) = d > 0$. Let $F, G \in S$ be such that  $F,\, G \notin PS$ but $FG \in PS$. By the induction hypothesis,
for every prime $Q \supset P$ of $R$ such that the height of $Q/P$ is one in $R/P$, $S/QS$ is a domain.   Hence,
$F \in QS$ or $G \in QS$.  Since $R$ is a Hilbert ring,  $P$ is an intersection of maximal ideals  $m$,
all of which contain such a $Q$. Hence, $P$ is the intersection of all such $Q$, and the family
of such $Q$ is infinite.  Thus, either $F$ or $G$, say $F$, is in $Q_iS$ for infinitely many
choices $\vect Q i\, \ldots$  of the prime $Q$.  Hence, $F \in \bigcap_{i=1}^\infty Q_iS = (\bigcap_{i=1}^\infty Q_i)S$,
because $R \to S$ is intersection flat.  But $\bigcap_{i=1}^\infty Q_i = P$,  since $f \notin P$ cannot have the
property that $f+P$ has infinitely many minimal primes in $R/P$.  Hence, $F \in PS$, a contradiction.
\end{proof}

Second, we observe:
%%2.8
\begin{proposition}\label{grprime}  Let $R$ be an $\N$-graded domain and let $\vect  F n$ be a regular sequence 
of forms that generate a prime ideal $P$.  Let $\vect f n$ be elements of $R$ whose leading forms are the elements
$\vect F n$.  Then $\vect f n$ generate a prime ideal $Q$. \end{proposition}

\begin{proof} Let $L(g)$ denote the leading form of $g \in R$.  

If the result is false we may first make a choice of $g,\, h \notin  Q$ such that $gh \in Q$  and such that, among all such choices, 
the degree of $gh$ minimum.  Second, for this choice of $g$ and $h$, among all ways of writing $gh = \sum_{i=1}^n r_if_i$, 
choose one such that the largest degree $\delta$ of any of the $L(r_if_i) = L(r_i)F_i$ is minimum.

If $\delta > \deg(gh)$, let $S$ be the set of indices $i$ such that $\deg(r_if_i) =\delta$.  Then $\sum_{i \in S} L(r_i)F_i = 0$, and 
the vector whose entries are the  $L(r_i)$ is a graded linear combination
of Koszul relations on the $F_i$, say, $\sum_{ij} h_{ij} (F_je_i - F_i e_j)$.  
We can  replace each $F_i$ by $f_i$ in this expression to obtain a relation on the $f_i$:  $\sum_i u_if_i = 0$.
Then $gh = \sum_{i=1}^n (r_i-u_i)f_i$ has a smaller value for $\delta$ on the right
hand side.  Hence, we may assume that $\delta = \deg (gh)$.  But then $L(g)L(h) \in P$,  and one of them,
say $L(g)$, is in $P$.  We may alter $g$ by subtracting a linear combination of the $f_i$ so as to cancel its leading form and 
so obtain $g'h \in Q$ with $g',\,h \notin Q$, contradicting the minimality of the degree of $gh$. \end{proof}
%%2.9  
 \begin{corollary}\label{expand}  Let $K$ be an algebraically closed field, and let $R = K[\vect g B]$ denote a polynomial
 ring over $K$.  Suppose $K[\vect g B] \inc L[\vect x N] = S$, a polynomial ring over a field $L$ such that
 the inclusion is graded and $\vect g B$ is a prime sequence in $S$.  Then for every prime ideal $P$ of
 $R$,   $PS$ is prime. \end{corollary}
 
 \begin{proof} By Proposition~\ref{grprime} above, for any  $\vect cB \in K$,  $g_1-c_1, \, \ldots, \, g_B-c_B$
 is prime in $S$.  The result is now immediate from Theorem~\ref{primeHilb}. \end{proof}
%%2.10
Corollary~\ref{expand} can also be deduced from \cite{EGA}, Th\'eor\`eme 12.1(viii). 

We also note the following fact, which is immediate from \cite{Serre}, Proposition 15, p.~IV--25.
 %%2.11
 \begin{proposition}\label{primary} Let $R$ to $S$ be flat extension of Noetherian rings, and let $M$ be a $P$-coprimary
 $R$-module, i.e.,  the set of associated primes of $M$ is $\{P\}$.  Suppose that $PS$ is prime.  Then
 $S \otimes M$ is $PS$-coprimary. \qed \end{proposition}

\section{Bounding all data for calculations with ideals or modules when the number of variables is known}\label{bound} 

The results of this section are expected, and likely can be deduced by nonstandard methods 
as in \cite{DS} or possibly even from \cite{Seid}, and they are closely related in both content and methods to those of \cite{EGA}, (9.8).   
However, what we need is not precisely given in any of those papers, and we give a brief treatment here that contains what we need 
for both this and subsequent papers.

 Let $R = K[\vect x B]$ be a polynomial ring over an algebraically closed field. When we refer to degrees we have in mind the standard  grading in which the variables have degree one.   
 But one may use a nonstandard positive integer grading instead, since the ratio of the two notions of degree
 is bounded above and below by constants.   Consider an $m \times n$ matrix
 $\cM$ with entries in $R$ such that the degrees of the entries are at most a given integer $d$.  Let $M \inc R^m$ be the 
 column space of $\cM$. 
In this section we show that bounds for the data of a primary decomposition of $M$ in $R^m$ (respectively, of a
finite free resolution of $M$)  can be given in terms of $B,\, m, n, d$, where the {\it data} of the decomposition include
the number of associated primes, the number of generators of each, and the number of generators and the degrees
of the generators of all the modules in the primary decomposition.  As will be evident from the proof, one can keep
track of more numerical characteristics.  By the {\it data} of a finite free resolution, we mean the length, the ranks of the
free modules modules occurring, and the degrees of the entries of the matrices.  The bounds are independent of the choice of $K$.  
We also obtain bounds for the operations occurring in Theorem~E when the number of variables
is bounded. 

The results of this section are  very different from other bounds obtained elsewhere in the paper, because
they are allowed to depend on $B$, the number of variables in the polynomial ring.   We shall apply
them in situations where we have a bound on $B$ that is independent of $K$ and $N$.

We give a brief overview of the key ideas in this section.  We are working with $n$ polynomials of
degree at most $d$ in $B$ variables over some algebraically closed field $K$.
We want to bound the numerical data associated
with with a finite free resolution or a primary decomposition of a module constructed as a cokernel of a matrix of given size (the
number of entries is at most $n$) whose entries are among 
these polynomials.  We want the bounds to involve only $n$, $d$ and $B$, and to be valid over any algebraically closed field $K$.  To achieve this, we replace
all of the coefficients by indeterminates over the integers  $\Z$.  Let $A$ denote the polynomial
ring over $\Z$ generated by all of these indeterminate coefficients.  The problem is then to bound the resolution or primary decomposition
after specializing the coefficients by applying an arbitrary homomorphism from $A$ to an algebraically closed field.  It is easier to approach the problem if
we think of $A$ as an arbitrary Noetherian domain, so that we have the freedom of replacing $A$ by a homomorphic image domain
and can assume, by Noetherian induction, that we know the result for the homomorphic image.  This means that the problem
reduces to constructing the needed bounds for all homomorphisms $A \to K$ whose kernels lie in a non-empty Zariski open subset $U$ of 
$X = \Spec(A)$.  The task then remains to construct the bounds for the all the irreducible components of the proper closed
set  $X - U$.  But these correspond to proper homomorphic image domains of $A$, and so we may apply Noetherian induction.

As indicated above, to construct the bounds for a non-empty open subset of $\Spec(A)$,  we start by constructing the resolution or primary
decomposition over the algebraic closure $\cF$ of $\fra(A)$.   This field is a directed union of module-finite extensions $A'$ of localizations
$A_a$ of $A$, where $a \not=0$.   By taking $A'$ sufficiently large, we can descend the resolution or primary decomposition over $\cF$ 
so that it only involves modules over  $A'$.  We can then make use of sufficiently (but finitely) many instances of
Grothendieck's lemma of generic freeness (in doing this, we localize at one more nonzero element of $A$)  so that the resolution or primary decomposition is preserved by base change from
$A'$ to any algebraically closed field.  If the new choice of $A'$ is module-finite over  $A_{a_1}$,  this solves the problem over
the open set $\Spec(A_{a_1})$. Note that any homomorphism from  $A_{a_1}$ to an algebraically closed field extends to
the module-finite extension $A'$.  The details are carried through in the remainder of this section.

%%3.1
\begin{theorem}\label{fixedB}  Let $h \geq 2$, $B$, $m$, $n$, $r$, and $d$ vary in $\N$.
Then there exist ascending functions $\cT(B,m,r,n,d)$, $\cG(B,m,n,d)$, $\cL(B,m,n,d,h)$, $\cE(B,m,n,d)$, and 
$\cP(B,m,n,d)$ with values in $\Z_+$ with the properties described below. 
Let $K$ be an algebraically closed field and let $R = K[\vect xB]$ be 
the polynomial ring in $B$ variables over $K$.   Let $m, n, d \in \N$, let $\cM$ be an $m \times n$ matrix over 
$R$ whose  entries have degree at most $d$, and let
$M$ be the column space of $\cM$.f
\begin{enumerate}[(a)]

\item Given an $r \times m$ matrix over $R$ with entries of degree at most $d$, thought of as a map from $R^r \to R^m$, 
and a set of generators for a submodule $M$ of $R^m$ bounded by $n, d$, there is a set of generators
for $\Ker (R^r \to R^m \surj R^m/M)$ bounded by $\cT(B,m,r,n,d)$.
\item There exist
sets of generators for $M \cap N$, $M:_R N$, $M:_G I$, and $M:_R I^{\infty}$ bounded by $\cG(B,m,n,d)$. 
\item There exists a set of generators for $M_1 \cap \cdots \cap M_h$ bounded by $\cL(B,m,n,d,h)$.
\item There exists an ascending function $\cE(B,m,n,d)$ that bounds
the number of primary components in an irredundant primary decomposition of $M$ in $R^m$, the number
of and the degrees of the generators of every prime and primary ideal occurring, and the number of generators 
and the degrees of the entries  
of the generators for every module in the decomposition. 

%%$\cE(B,m,n,d)$ can also be taken to bound the exponent
%%on every associated  prime ideal  $P$ needed to annihilate the corresponding $P$-coprimary
%%component of $M$ mod $M$ (in the ideal case, the exponent $a$ needed so that $P^a$ is contained in the corresponding %%primary ideal of the decomposition). 

\item There exists an ascending function $\cP(B,m,n,d)$  that bounds, independent of $K$, 
the length a  free resolution of $M$, the ranks of the free modules occurring, and the degrees 
of all of the entries of all of the matrices occurring.  In the graded case, $\cP(B,m,n,d)$ bounds the twists of $R$ that
occur as summands in a minimal free resolution of $M$.

\end{enumerate}
\end{theorem}
%%3.2
\begin{discussion}\label{idea} As already indicated in the introductory paragraphs of this section, we shall prove this result by first considering the case where all the entries of the matrices occurring and all entries of generators of the modules and ideals occurring are replaced by generic polynomials of degree at most $d$ with distinct coefficients $u_t$ that are variables over $\Z$.  In this discussion we give
further detail.

 Let $A$ denote the polynomial ring over $\Z$ in all the variable coefficients.  Then we can cover
$\Spec(A)$ by a finite number of locally closed affines $\Spec(A_s)$ for each of which there
is a generic calculation of the kernel, intersection, colon, primary decomposition of the module, or a generic finite free resolution
(for these,  each $A_s$ is replaced by a module-finite extension domain $A_s'$).  These generic  calculations
specialize to give all ones needed
when we replace the variables $u_t$ by elements of an algebraically closed field $K$:  when we
make that replacement, we obtain a map $A \to K$ whose kernel  $P \in \Spec(A)$ lies in one of the
$\Spec(A_i)$ for $i \geq 1$,  and the required calculation over $K$ is
obtained by extending the map $A_i \to K$ to a map $A_i' \to K$, and then tensoring over
$A_i'$ with $K$.  A more complete explanation is given below.

To carry through this idea, we first do the generic calculation or  primary decomposition or free resolution over an
open affine $\Spec(A_1)$ in $\Spec(A)$.  The complementary closed set is a union of closed irreducibles.
We can then iterate the procedure with each of these irreducible closed sets.  We carry out this
construction when $A$ is  an arbitrary Noetherian domain, and the results we need will
follow readily once we have carried through the first step, i.e., once we have shown that
we can find an open affine $A_1$ and a module-finite extension $A_1'$ where there is a generic calculation, or 
primary decomposition, or free resolution.  It then follows by Noetherian induction that for each irreducible
component of the complement of $A_1$, one already has a finite cover by locally closed affines as
described.  Discussion~\ref{generic} just below together with Proposition~\ref{genops} construct $A_1$ for an 
arbitrary Noetherian domain $A$.\end{discussion}

%%3.3
\begin{discussion}\label{generic}
Let $A$ be a Noetherian domain and let $R_A = A[\vect x B]$.
Let $\cQ_A$ (respectively, $\cM_A$) be a $r \times m$ (respectively, $m \times n$) matrix over $R_A$  and let
$M_A$ be the column space of $\cM_A$.  Let $I_A$ be an ideal of $R_A$, and let $N, \, M_{1,A}, \ldots, \, , M_{h,A}$
be submodules of $G_A := R_A^m$.  Let $W_A$ be the kernel of the composite map 
$$R_A^r \to R_A^m \surj R_A^m/M_A,$$ where the map on the left has matrix $\cQ_A$. 
For any $A$-algebra $S$, let $R_S$, $M_S$, $G_S$, etc. denote the tensor products 
over $A$ of $R_A$, $\cM_A$ with $S$. In the case of $\cM_A$ or $\cN_A$,  $\cM_S$ or $\cN_S$ is the result of replacing each
entry of the matrix considered by its image in $S$.  
Let $\cF$ denote an algebraic closure of the fraction field of $A$.  In the case of ideals $I_A$ or submodules of $M'_A$
of  $G_A$, the change
in subscript from $A$ to $S$ indicates that $I_A$ or $M_A$ is to be replaced by its image in $R_A$ or $G_A$.

A well-known form of generic flatness (perhaps more accurately, generic freeness) asserts that if $W_{A}$ is a finitely 
generated $S_{A}$-module, where
$S_A$ is a finitely generated algebra over a Noetherian domain $A$, one can localize at one element of 
$a \in A-\{0\}$ so that $(W_A)_a$ is $A_a$-free. It is also true that if $T_A$ is a finitely generated $S_A$-algebra,
that $W_A$ is a finitely generated $T_A$-module, and $Q_A$ is a finitely generated $S_A$-submodule of $W_A$,
one may localize at one element of $A$ so that $(W_A/Q_A)_a$ is $A_a$-free:  see Lemma 8.1 of \cite{HR}.  Of course,
in applying this we may take $S_A$ to be $R_A$.  

Note that for $I_A \inc R_A$ or $M'_A \inc G_A$ there are two possible meanings for $I_S$ and $M'_S$:  one
is $I_A \otimes_A S$ (respectively, $M'_A \otimes_A S$) and the other is its image in $R_S$ (respectively,
$G'_S$).  By the theorem on generic freeness, using the image will be the same as the result of tensoring with  $S$ if 
we first replace $A$ by a suitable localization at one element of $A-\{0\}$,  which we will be free to do in this section, and
we shall assume that $A$ has been replaced by such a localization for which the two agree.  

Let $\cA$ denote the family of extension
rings of $A$ within $\cF$ obtained by localizing at
one element of $A-\{0\}$ and then adjoining finitely many integral elements of $\cF$.  (The same rings may
be obtained by adjoining finitely many integral elements of $\cF$ to $A$ and then localizing at one element of $A-\{0\}$.)  
Note that $\cF$ is the directed union of the rings in $\cA$.  

For each of  $W_\cF$ ($W_A$ is defined above as a certain kernel), $M_{1,\cF} \cap \cdots \cap M_{h,\cF}$,
$M_{\cF} :_{R_\cF} Q_\cF$,  $M_\cF:_{G_\cF} I_{\cF}$, 
$M_{\cF}:_{G_{\cF}}I_{\cF}^\infty$,
a chosen irredundant primary decomposition of $M_\cF$ in $G_\cF$,  and a chosen finite free resolution of
$\cM$ over $R_\cF$,  one can choose $A' \in \cA$, module-finite
over $A_1 = A_a$, such 
that the kernel, intersection, colon, primary decomposition or finite free resolution is defined over $A'$.
It may not have the same property over $A'$, but that can be restored after localizing at one nonzero
element of $A$. 

%%3.4
\begin{proposition}\label{genops} Let notation and hypotheses be as in Discussion~\ref{generic} just above. 
After localizing at one more nonzero element of $A$, we have a calculation of the kernel, intersection,
or colon, or a primary decomposition or finite free resolution over $A' \in \cA$  which is preserved by
arbitrary base change to an algebraically closed field $K$.  Since every map $A_1 \to K$,  where $K$ is an
algebraically closed field, extends to a map $A' \to K$,  the kernel, intersection, colon, primary decomposition
and finite free resolution over $K$ arise  from the one over $A'$ by specialization, i.e., by base change from
$A'$ to $K$. \end{proposition}

\begin{proof}  
We shall be applying generic freeness repeatedly with $A'$ replacing $A$.  Since every nonzero element of $A'$ 
has a nonzero multiple in $A$,  we may assume in these applications that we are localizing at an element of $A-\{0\}$.  
We may localize at one
element of $A$ and achieve a finite number of instances of freeness over $A'$.

First note that after localizing at one element of $A-\{0\}$, we can preserve the
exactness of a finite number of short exact sequences of finitely generated $R_A$-modules upon tensoring
with {\it any} $A'$-algebra $L$.  We may also preserve the inclusions in a finite filtration of a finitely
generated $R_{A'}$-module,  as well as the injectivity of an $A'$-algebra map $S_{A'} \to T_{A'}$ of finitely
generated $A'$-algebras upon tensoring with an arbitrary $A'$-algebra $L$ over $A'$.  
We may preserve intersections of two (hence,
finitely many) submodules $M_{A'}$, $N'_{A'}$ of a finitely generated $R_{A'}$ module  $W_{A'}$,
because we may preserve the inclusions of $M_{A'}$ and $N_{A'}$ in $W_{A"}$
as well as the exactness of the sequence 
$$0 \to M_{A'} \cap N'_{A'} \to M_{A'} \oplus N_{A'} \to M_{A'} + N_{A'} \to 0$$ under arbitrary base 
change by localizing at one element of $A-\{0\}$ so that all the modules involved become $A'$-free.

With these remarks it is obvious that we can preserve the exactness of
$$0 \to W_A' \to R_{A'}^r \to R_{A'}^m/M_A \to 0$$
under any base change $A' \to L$.  We already know that we can preserve finite intersections of submodules.
If $\vect  u t$ generate $N_A$,  we have an exact sequence 
$$0 \to M_{A'}:R_{A'} N_{A'} \to R_{A'} \to (G_{A'}/M_{A'})^{\oplus t}$$
where the image of $r \in R_A$ is the vector whose $t$ entries are the images of the elements $ru_j$ in $G_{A'}/M_{A'}$.
Likewise, if $\vect f t$ generate $I_{A'}$, we have an exact sequence 
$$0 \to M_{A'}:_{G_{A'}} I_A \to G_{A'} \to (G_A/M_A)^{\oplus t}$$ where the image of $u \in G_{A'}$
is the image of the vector $(f_1u, \, \ldots, \, f_tu)$.  
 
We now consider primary decomposition. 
We can preserve that an element of a module  (hence, the module itself) is nonzero, if that is true
after tensoring with $\cF$.   Call the element
$u_A$ and the module $Q_A$.  We may localize so that all the terms of $0 \to R_Au_A \to Q_A \to W_A \to 0$
become $A$-free,  and $R_Au_A$ is then a nonzero free $A$-module.  Then  $R_Lu_L$ is nonzero
for every nonzero $A$-algebra $L$,  and injects into $Q_L$.   This enables us to keep modules distinct, and to
keep ideals distinct. 

If a primary decomposition of $M_{A'}$ in $G_A$ is irredundant, this can be preserved:  we can localize sufficiently
that intersection commutes with base change for all finite sets of primary components, and we can keep every intersection 
that omits a component distinct from $M_{A'}$.  Likewise, we can keep all the primes that occur distinct.  
We need an additional argument to show that the primes remain primes and that components remain primary.

If $P_{A'}$ is such that $P_\cF$ is prime, we can localize at one element of $A-\{0\}$ and guarantee
that $P_L$ is prime for every map of $A'$ to an algebraically closed field $L$.  In fact, it suffices to preserve that 
$D_{A'} = R_{A'}/P_{A'}$
is a domain for a finitely generated $A'$-algebra $D_{A'}$, given that $D_\cF$ is a domain.  After localizing
at one nonzero element of $A$, we have that $D_{A'}$ is module-finite over a polynomial ring over $A'$.
After enlarging $A'$ and $D_{A'}$ by adjoining finitely many $p^e\,$th roots of elements of $A'$ and of the variables,
we may assume that the $D_{A'}$ is contained in a domain $D'_{A'}$ obtained by making a separable extension of
the fraction field of a polynomial ring over $A'$ and adjoining finitely many integral elements in that separable extension.  
By the theorem on the primitive element for separable field extensions, $D'_{A'}$ has the same fraction field
as $A'[\vect x h][\theta]$ where $\theta$ satisfies a monic irreducible separable polynomial $H_A'$
over $A'[\vect x h]$.  Now we an choose $G_{A'} \in A'[\vect xh]$ such that 
$D'_{A'} \inc (A'[\vect x h]_{G_{A'}})[\theta]$. By inverting an element of $A-\{0\}$ we may assume that $G_{A'}$ is monic. 
To complete the proof, it suffices to show that we can, after enlarging $A'$,  
keep the minimal polynomial $H_{A'}$ of $\theta$ (which we may also assume is monic) irreducible no 
matter what field we tensor with.  This can be done using Hilbert's Nullstellensatz. For each positive degree $s$ 
strictly smaller
than the degree of $H_{A'}$,  we can write down a potential factorization of $H_{A'}$, namely  $H_{A'} = H'_{A'}H''_{A'}$,  
where
$H'_{A'}$ has degree $s$, and we use indeterminates for all the coefficients of $H'_{A'}$ and $H''_{A'}$.
Equating corresponding coefficients yields a system of polynomial equations in the unknown coefficients $Z_j$.
We know these equations have no solution in the algebraically closed field $\cF$.  Hence, the polynomials
we are setting equal to 0 generate the unit ideal in $\cF[Z_j:j]$.   They will therefore still generate the unit ideal
in $A'[Z_j:j]$   for a suitably large choice of $A'$.

We can preserve that a submodule is $P_{A'}$-coprimary:  filter the module by torsion-free modules over
$R_{A'}/P_{A'}$, and embed each in a free $(R_{A'}/P_{A'})$-module.  The filtration and the embedding will
be preserved by arbitrary base change after localization at a suitable element of $A-\{0\}$. 

Thus, we can choose a primary decomposition over $A'$ that is preserved by base change to any
algebraically closed field. 

To preserve $M_{A'}:I_{A'}^\infty$ under base change, we note that this module is the same as the intersection
of those primary components of $M_{A'}$ such that the corresponding prime does not contain $I_{A'}$.  

It is clear that one can preserve a finite free resolution:  its exactness is equivalent to the exactness
of finitely many short exact sequences.  
  \end{proof}
  
\end{discussion}
  
We use Discussion~\ref{generic} above to construct the
open affine $A_1$.  As mentioned earlier, we now obtain the required cover by locally closed open affines by applying
Noetherian induction to the irreducible components of the complement of $\Spec(A_1)$ in $\Spec(A)$.  \\

\noindent{\it Proof of Theorem~\ref{fixedB}.} 
By applying this procedure to $A$ as defined in Discussion~\ref{idea}, we obtain finitely many kernels, colons, intersections, primary
decompositions or finite free resolutions that give rise to all others needed over any
algebraically closed field by specialization.  The existence of the bounds stated in Theorem~\ref{fixedB}
is immediate. \qed

\begin{remark}  It is clear from the argument that we can bound much more if we choose to, by taking
a finer stratification.  For example, we can bound all the data associated with finite free resolutions
of the ideals and/or modules in the primary decomposition, and the same is true for finitely many other
ideals and/or modules formed from them by iterated intersection, colon, product, and sum.  \end{remark}

%%4
\section{The proof of the main  theorems A, B, C, D, E and F}\label{proof}

We shall prove that if Theorems A, B, C, D, E, and F  
hold for positive integers strictly less than $d$ then they hold also for degree $d$.  We note that all of the
theorems are obvious if $d = 1$.

To prove part (a) of Theorem A,  let $D: =  D(k-1,d-1)$, which bounds the number of generators
of a minimal prime of an ideal generated by a regular sequence of $k-1$ or fewer forms of degree $d-1$,
and which exists by the induction hypothesis.  Let $\Phi$ be as in Proposition~\ref{phi},  which also exists
by the induction hypothesis, since one has a function ${}^3\cB(n,d-1)$ as in Corollary B.  If the strength of 
the $d$-form $F$ is at least $\Phi(D,d)$, but the height of $(\cD F)R$ is at most $k-1$,  we can choose $k-1$  or fewer polynomials in $\cD F$ that form a maximal regular
sequence, and we can choose an associated (equivalently, minimal) prime of the ideal they generate that contains
$\cD F$.  The number of generators is bounded by $D= D(k-1,d-1)$.  
Hence, using only those generators of degree at most $d-1$,  we obtain that $\cD F$ is contained in an ideal $J$ 
generated by at most $D$ forms of degree $\leq d-1$.  
By Theorem F,  this contradicts the strength assumption on $F$.  In characteristic 0 or  $p > d$,
we could simply have assumed that $F$ has strength $D$.  

Part (b) of Theorem A follows from part (c). We use induction on $n$.
The result is clear if $n=1$.  We may assume $n >1$ and that any $n-1$ or fewer linearly independent 
homogeneous elements in $V$ form an $\rom{R}_\eta$-sequence. None of the elements in the basis
is in the ideal generated by the others:  if it were, we would get a graded relation on the basis elements in which
one of the coefficients is 1:  say it is the coefficient of an element of degree $i$.  Then a nonzero linear
combination of elements of degree $i$ has an $k$-collapse for $k \leq i-1$, a contradiction, since we
are assuming $\etA(i) \geq i-1$. 
Since the quotient by $n-1$ or fewer elements in the basis satisfies R$_\eta$, and so is a domain,
we may assume that a basis for $V$ consisting of forms is a regular sequence.  

From the property of the $\etA(i)$ stated in part (a) of Theorem A, each row of the Jacobian matrix of a basis
for $V$ with respect to $\vect x N$ generates an ideal of height $\eta + 3n-3 + 1 + 1 = \eta + 3n-1$ in the  
 polynomial ring.  Since $n \geq h$, where $h$ is the number of nonzero $V_i$,  
 by Theorem~\ref{combined}, the height of the defining ideal of the singular locus 
 of $R/(V)$ in $R/(V)$ is at least $\eta + 1$, so that $R/(V)$ satisfies $\rom{R}_\eta$. 

We next show that Theorem A in degree at most $d$ implies Theorem B in degree at most $d$.  Linearly order
the dimension sequences $\delta = (\vect \delta d)$ so that  $\delta < \delta'$ precisely if $\delta_i < \delta_i'$ for
the largest value of $i$ for which the two are different.  This is a well-ordering.  Assume that $\etB$ is defined for 
all predecessors of $\delta$.  If the vector space is $\etA(\delta)$-strong,  it satisfies $\rom{R}_\eta$
and we are done.  If not, for some $i$ an element of $V_i$ has an $\etA_i(\delta)$-collapse, and we can
express the element using at most $2\cdot\etA_i(\delta)$ forms of lower degree. This enables us
to form a new vector space in which $\delta_j$ remains the same for $j > i$, $\delta_i$ decreases
by 1,  and the $\delta_j$ for $j <i$  increase by a total of $2\cdot\etA_i(\delta)$.  If we let $\delta'$ run through
all dimension sequences, with this property, that precede $\delta$ in the well-ordering, we may take  
$\etB(\delta) = \max_{\delta'}\{\etB(\delta')\}$. This completes the proof of Theorem B.

Theorem C is immediate, because if $d$ bounds the degrees of the entries of the matrix,
then $mnd$ bounds the number of non-scalar homogeneous components of all entries, and  $C(m,n,d):= 
\max\{B(\delta): \sum_{t=1}^d\delta_t = mnd\}$ bounds the projective dimension of the cokernel. 

Theorems D and E follow at once from the existence of $\etB$ and Theorem~\ref{fixedB} of the 
preceding section, while as already noted, Theorem F in degree $d$ follows from Theorem B in
degree $d-1$ by Proposition~\ref{phi}.
\qed

\quad\bigskip

$\begin{array}{ll}
\textrm{Department of Mathematics}           &\qquad\qquad \textrm{Altair Engineering}\\
\textrm{University of Michigan}                    &\qquad\qquad \textrm{1820 E.\ Big Beaver Rd.}\\
\textrm{Ann Arbor, MI 48109--1043}            &\qquad\qquad \textrm{Troy, MI 48083}\\
\textrm{USA}                                                &\qquad\qquad \textrm{USA}\\
\quad & \quad\\
\textrm{E-mail: hochster@umich.edu}          &\qquad\qquad \textrm{E-mail: antigran@umich.edu}\\ 
        
\end{array}$

\end{document}